\documentclass[12pt]{amsart}
\usepackage{amssymb, latexsym, amsthm} %, makeidx}

\newtheorem{thm}{Theorem}[section] %the resolution could also be [subsection]
\newtheorem{cor}[thm]{Corollary}
\newtheorem{defn}[thm]{Definition}

\newtheorem{lem}[thm]{Lemma}
\newtheorem{prop}[thm]{Proposition}

\newtheorem{rem}[thm]{Remark}

\newcommand\Cref[1]{{Corollary~\ref{#1}}}

\newcommand\Lref[1]{{Lemma~\ref{#1}}}
\newcommand\Pref[1]{{Proposition~\ref{#1}}}

\newcommand\Tref[1]{{Theorem~\ref{#1}}}

\def\Z {{\mathbb {Z}}}

\newcommand\mul[1]{{#1^{\times}}} % The multiplicative group
\newcommand\operA[2]{{\if!#2!\operatorname{#1}\else{\operatorname{#1}_{#2}^{\phantom{I}}}\fi}} 

\long\def\forget#1\forgotten{}
\renewcommand\H[4][!]{{\operatorname{H}^{#2}\!\!\;({#3},{#4}{\if!#1\relax\else(#1)\fi})}}

\DeclareMathOperator{\Cor}{cor}
\DeclareMathOperator{\I}{Im}
\DeclareMathOperator{\K}{Ker}

\DeclareMathOperator{\pr}{Proj}

\DeclareMathOperator{\BKer}{BKer}
\DeclareMathOperator{\3MP}{3MP}

\DeclareMathOperator{\ch}{char}

\DeclareMathOperator{\Br}{Br}

\DeclareMathOperator{\Tr}{Tr}

\DeclareMathOperator{\N}{N}

\DeclareMathOperator{\Zp}{\Z/p\Z}
\DeclareMathOperator{\Ker}{Ker}
\DeclareMathOperator{\ram}{ram}

\DeclareMathOperator{\Gal}{Gal}

\def\codim{{\operatorname{codim}}}
\def\coker{{\operatorname{coker}}}

\def\N{{\operatorname{N}}}

\def\Z{\mathbb{Z}}

\newcommand\res[1][{}] {{\operatorname{res}_{#1}}}
\newcommand{\Trace}[1][]{\if!#1!\operatorname{Tr}\else{\operatorname{Tr}_{#1}^{\phantom{I}}}\fi} % Usage: $\Tr[K/F]\chi_a$.

\newcommand\book[4]{{{#1},\ {{#2}}{\if!#3!\relax\else{,\ {#3}}\fi}{\if!#4!\relax\else{,\ {#4}}\fi}.}} % Format: \book{author}{title}{info}{year}

\newif\ifXY % turns XY version on/off
\XYtrue     % Turn it on
\ifXY

\input xy
\input xyidioms.tex
\usepackage{xy}
\xyoption{all} %
\fi % For \ifXY

\begin{document}
\title{Triple Massey products with weights in Galois cohomology}
\author{Eliyahu Matzri}

\address{Department of Mathematics \\
         Bar-Ilan University\\
         Ramat-Gan \\
         Israel}
\email{elimatzri@gmail.com}
%\thanks{The author was supported by the Israel Science Foundation (grant No.\ 152/13) and by the Kreitman foundation.}
\thanks{The author would like to thank I. Efrat for introducing him with Massey products and for many interesting discussions and helpful suggestions.}

\begin{abstract}
Fix an arbitrary prime $p$. Let $F$ be a field containing a primitive $p$-th root of unity, with absolute Galois group $G_F$, and let $H^n$ denote its mod $p$ cohomology group $H^n(G_F,\Z/p\Z)$.
The triple Massey product of weight $(n,k,m)\in \mathbb{N}^3$ is a partially defined, multi-valued function $\langle \cdot,\cdot,\cdot \rangle: H^n\times H^k\times H^m\rightarrow  H^{n+k+m-1}.$ %(in the mod-$p$ Galois cohomology)
In this work we prove that for an arbitrary prime $p$, any defined $\3MP$ of weight $(n,1,m)$, where the first and third entries are assumed to be symbols, contains zero; and that for $p=2$ any defined $\3MP$ of the weight $(1,k,1)$, where the middle entry is a symbol, contains zero. Finally, we use the description of the kernel of multiplication by a symbol to study general 3MP where the middle slot is a symbol. The main tools we will be using is Lemma \ref{L1} concerning the the annihilator of cup product with an $H^1$ element, and \Tref{Tig}, generalizing a Theorem of Tignol on quaternion algebras with trivial corestriction along a separable quadratic extension.
\end{abstract}
%Lemma concerning the the annihilator of cup product with an $H^1(F)$ element using the notion of a Norm variety.

\maketitle
%-----------------------------------------------------------------
% \setcounter{tocdepth}{3}
% \tableofcontents
%

%\begin{acknowledgment}
%\end{acknowledgment}
\section{introduction}
Let $p$ be a prime number. Let $F$ be a field of characteristic prime to $p$, containing a primitive $p$-th root of unity, with absolute Galois group $G_F$, and let $H^n$ denote its mod $p$ cohomology group $H^n(G_F,\Z/p\Z)$.
In recent years there has been increasing interest in specific external operations called $d-$fold Massey products on $H(G_F)=\oplus H^n$, the cohomology $\Z/p\Z$-algebra of the differential graded $\Z/p\Z$-algebra (abbreviated DGA) $C(G_F)=\oplus C^n$ of continuous cochains on $G_F$. For $d\geq 2$, the $d$-fold Massey product is a family, indexed by $\mathbb{N}^d$, of certain semi-defined, multi-valued functions, which we denote dMP of weight $(n_1,\cdots,n_d)$ (for a precise definition we refer the reader to \cite{Kraines66}). In particular for $d=3$, the $\3MP$ of weight $(n,k,m)$ is a function as above,  $$\langle \cdot,\cdot,\cdot \rangle: H^n\times H^k\times H^m\rightarrow  H^{n+k+m-1}.$$
A defined dMP is called \textbf{essential} if it does not contains the zero element. One of the main conjectures (due to Min\'{a}\v{c} and T\^{a}n \cite{MT}) in this context is that in Galois cohomology for $d\geq 3$, there are no essential dMP of weight $(1,1,\cdots ,1)$.
Recently, the case of $\3MP$ of weight $(1,1,1)$ was the focus of attention. Starting with the case of number fields and $p=2$, in \cite{HW}, Hopkins and Wickelgren proved that there are no essential $\3MP$ of weight $(1,1,1)$ in Galois cohomology, by constructing a splitting variety for the $\3MP$ and using a local-global principle to show it has a rational point. Then in \cite{MT}, Min\'{a}\v{c} and T\^{a}n generalized the result to arbitrary field and $p=2$ by introducing a rational point over any field, and later in \cite{MT2}, to number fields and arbitrary prime, using a local-global principle for certain embedding problems. In \cite{EM} the above results were obtained by translating the question to the theory of central simple algebras and using the classical Theorem of Albert, Brauer, Hasse and Noether for number fields and that of Albert on bi-quaternion algebras. Finally, the general case of the conjecture for $\3MP$ of weight $(1,1,1)$ was proved in \cite{EM2}, \cite{Mat14} and generalized to the case where the base field does not contain a primitive $p$-th root of unity in \cite{MT14c}.

In this work we consider higher $\3MP$ where some slots are assumed to be symbols.
The work is organized as follows: In section $2$, we give the necessary background on triple Massey products. In section $3$ we give some background on the Norm variety of a non trivial symbol and prove the characteristic zero case of  \Lref{L1}, stating: \\
\noindent\textbf{\Lref{L1}.}\emph{ Let $F$ be a field of characteristic prime to $p$, which is prime to $p$ closed. Let $\alpha \in H^n(F)$ be a symbol and $b\in H^1(F)$. Then $\alpha \cdot b=0$ if and only if there exist $s_i\in H^1(F) \  i=1,\dots, n$ and a presentation $\alpha=s_1\cdots s_n$ such that $s_n\cdot b=0$.}\\
In section $4$ we use \Lref{L1} to prove the characteristic zero case of \Tref{Tig} and \Tref{n1m} stating:\\
\noindent\textbf{\Tref{Tig}.}\emph{ Let $F$ be a field with $\ch(F)\neq 2$, and $K/F$ a field extension of degree $2$. Let $\alpha=\alpha'\cdot k\in H^{n+1}(K,\Z/2\Z)$ where, $\alpha'~\in~H^n(F)$ is a symbol and $k\in H^1(K)$. Then, $$\alpha=\res_{K/F}(\alpha'\cdot c) \ ; c\in H^1(F) \ \hbox{if and only if} \  \Cor_{K/F}(\alpha)=0.$$}\\
\noindent\textbf{\Tref{n1m}.}\emph{ Let $F$ be a field of characteristic prime to $p$, and $\langle \alpha, a, \beta \rangle$ be a defined $\3MP$ of weight $(n,1,m)$ where $\alpha, \beta$ are symbols. Then, $~0~\in~\langle \alpha, a, \beta \rangle.$}\\
In section $5$ we prove the characteristic zero case of \Cref{1n1} and \Tref{BKer} stating:\\
\noindent\textbf{\Cref{1n1}.}\emph{ Let $F$ be a field of characteristic prime to $p$, and $\langle x,\alpha, y \rangle $ be a defined $\3MP$ of weight $(1,n,1)$ where $\alpha$ is a symbol. Then, $0~\in~\langle x,\alpha, y\rangle.$ }\\
\noindent\textbf{\Tref{BKer}.}\emph{ Let $F$ be a field of characteristic prime to $p$, $\beta\in H^k$ a symbol, and $\alpha\in H^n\cap \BKer(\beta), \gamma\in H^m\cap \BKer(\beta)$. Then the 3MP $\langle \alpha, \beta, \gamma \rangle$ is defined and for any presentation $\alpha=\sum_{i=1} ^{s} \alpha_i; \gamma=\sum_{j=1} ^{t} \gamma_j$ where the summands are basic elements of $\Ker(\beta)$ we have,
$$\langle \alpha, \beta, \gamma \rangle \subseteq \sum_{i,j}(\alpha_i\cdot H^{k+m-1}+\gamma_j\cdot H^{n+k-1}).$$}
In section $6$ we prove that all the above is also valid for fields of positive characteristic prime to $p$. The reason for this section is that (to my knowledge) Norm varieties were only shown to exist in characteristic zero.

\section{Notations and conventions}

For a field $F$ and a prime number $p$, we use the following notations  \begin{itemize}

                          \item $G_F=\Gal(F^{sep}/F)$, where $F^{sep}$ is a fixed separable closure for $F$.
                          \item We say that $F$ is prime to $p$ closed if $F$ has no finite field extensions of degree prime to $p$. Equivalently, $G_F$ is a pro-$p$ group.
                          \item $C^n(F)=C^n(G_F,\Z/p\Z)$ (or just $C^n$ when $F$ is understood) is the continues $n$-cochains of $G_F$ and $C(F)=\oplus_{n\in \mathbb{N}}C^n(F)$ is the differential graded $\Z/p\Z$ algebra of continues cochains of $G_F$ with product being the cup product, which we denote by $\cdot$.
                          \item $H^n(F)=H^n(G_F,\Z/p\Z)$ is the $n$-th continues cohomology group of $G_F$, and $H(F)=\oplus_{n\in \mathbb{N}}H^n(F)$ is the cohomology $\Z/p\Z$ algebra.
                          \item We let $\res_{K/F}$ denote the restriction homomorphism $$\res_{K/F}~:~C(F)~\to~C(K)$$ and the induced homomorphism $$\res_{K/F}:H(F)\to H(K).$$
                          \item We let $\Cor_{K/F}$ denote the corestriction homomorphism $$\Cor_{K/F}:C(K)\to C(F)$$ and the induced homomorphism $$\Cor_{K/F}:H(K)\to H(F).$$
                          \item An element $\alpha\in H^n(F)$ is called a symbol if $\alpha=a_1\cdot a_2\cdots a_n$ for $a_i\in H^1(F), \ i=1\dots n.$
                          \item Whenever we talk about Massey products over a field $F$, we assume $F$ contains a root of unity of order $p$.
                 \end{itemize}

For $X/F$ a smooth, irreducible, projective variety of dimension $d$, one defines:
\begin{itemize}
\item The group of $0$-dimensional $K_1$-cycles as the cokernel of the residue homomorphism, $$A_0(X,K_1)=\coker \big{(}\underset{\codim (x)=d-1}{\coprod}K^M_{2}(F(x))\rightarrow \underset{\codim (x)=d}{\coprod}K^M_{1}(F(x))\big{)}.$$
\item The group of reduced $0$-dimensional $K_1$-cycles as $$\bar{A}_0(X,K_1)=\coker \big{(}A_0(X \times X,K_1)\overset{(pr_1)_{*}-(pr_2)_{*}}{\xrightarrow{\hspace*{2cm}}}A_0(X,K_1)\big{)},$$ for a detailed account of this we refer the reader to \cite{SJ}, \cite{Weib}.
\end{itemize}

\section{General triple Massey products}
In this section we give the necessary background on triple Massey products (for more details we refer the reader to \cite{Dwyer75}, \cite{IE}, \cite{HW}).
\subsection{Definitions}
%The main sources for this background subsection are \cite{Dwyer75}, \cite{IE}, \cite{HW} and \cite{Wic}.
Let $R$ be a unital commutative ring. A differential graded algebra (abbreviated DGA) over $R$ is a graded $R$-algebra
$$ C^\bullet = \oplus_{k\geq 0}C^k = C^0 \oplus C^1\oplus C^2\oplus ... $$
equipped with a differential $\delta: C^{\bullet} \rightarrow C^{\bullet+1}$~such~that:
\begin{enumerate}
\item $\delta(xy)=\delta(x) y+(-1)^k x\delta(y)$ for $x\in C^k;$
\item $\delta^2=0.$
\end{enumerate}
One then defines the cohomology algebra $H^{\bullet}=\K(\delta)/\I(\delta)$ of the DGA $C^{\bullet}.$
\begin{defn}
Let $(x_1,x_2,x_3)\in H^n\times H^m\times H^k$ be such that $$x_1x_2=0 \hbox{ and} \  x_2x_3=0.$$ Choose representatives, $(a_{1}, a_{2}, a_{3})\in C^n\times C^m\times C^k$ for $(x_1,x_2,x_3)$. Then there exists $(a_{1,2},a_{2,3})\in C^{n+m-1}\times C^{m+k-1}$ such that
$$\delta (a_{1,2})=a_1a_2 \ \hbox{and } \ \delta(a_{2,3})=a_2a_3,$$
such a collection of cochains, denoted $A$, is called a defining system for the $\3MP$ $\langle x_1, x_2, x_3\rangle$.
Given a defining system $A$, define the related cocycle $$c(A)=a_1a_{2,3}+(-1)^{n+1}a_{1,2}a_3,$$ note that it is a cocycle as
$\delta(c(A))=(-1)^na_1a_2a_3+(-1)^{n+1}a_1a_2a_3=0$.\\
Finally, define $\langle x_1, x_2, x_3\rangle $ to be the subset of $H^{n+m+k-1}$ consisting of all classes $w$ for which there exists a defining system $A$ such that $c(A)$ represents $w$.
\end{defn}
\begin{rem}
In \cite{Kraines66} the defined $\3MP$ is the same as ours but multiplied by the sign $(-1)^{n+m}$, in particular the following Theorem (taken from \cite{Kraines66}) holds for our $\3MP$.
\end{rem}

%\subsection{Properties of triple Massey products}
%Let $G$ be a profinite group. In [Krains] he proves properties regarding Higher Massey products in the context of the DGA of singular cochains of a pair of topological spaces $X,A$. We start by explaining why the results in [Krains] apply to the case of the DGA of continues cochains of $G$.
%Indeed if $G$ is finite, then it has the discrete topology and in this case it is known [reference] that its group cohomology is isomorphic to the singular cohomology of it's classifying space $BG$ (which has $G$ as its fundamental group), thus we can apply [Krains] to this case.
%As for the general case it follows from the fact that $G$ is the inverse limit of its finite quotients, reducing everything to the finite case.\\
%Let $G$ be a profinite group and $C(G)=\oplus C^n(G,\Z/p\Z)$, the DGA of continues cochains on $G$ with cup product as multiplication.
%In this subsection we summarize what we will need from \cite{Kraines66}.

%****reference to the above should be at chapter II of : Adem, Alejandro; Milgram, R. James (2004), Cohomology of Finite Groups, Grundlehren der Mathematischen Wissenschaften 309 (2nd ed.), Springer-Verlag, ISBN 3-540-20283-8, MR 2035696, Zbl 1061.20044\\ \ \\% \ \\

\begin{thm}\label{KT1}%Theorem 1:\\
Assume $\langle u_1,\dots ,u_k\rangle $ is defined and let
$v\in H^m$. Then the $k$-fold product $\langle u_1,\dots,u_t v, \dots, u_k\rangle $ is defined, for
$t=1,\dots, k.$ Furthermore the following relations are satisfied.\\
(1) $\langle u_1,\dots ,u_k\rangle v\subset \langle u_1,\dots ,u_k v\rangle $\\
(2) $v \langle u_1,\dots ,u_k\rangle \subset (-1)^{m}\langle v u_1,\dots ,u_k\rangle $\\
(3) $\langle u_1,\dots,u_t v, \dots, u_k\rangle \cap \langle u_1,\dots,u_t ,vu_{t+1}, \dots, u_k\rangle \neq \phi$\\
These relations may be interpreted as equalities modulo the sum of the indeterminacies.
\end{thm}
%\begin{thm}\label{KT2}%Theorem 2:\\
%Assume $\langle u_1,\dots ,u_k\rangle $ is defined. Then $\langle u_k,\dots ,u_1\rangle $ is defined and
%$$\langle u_1,\dots ,u_k\rangle = (-1)^h\langle u_k,\dots ,u_1\rangle$$
%where
%$h= \sum_{1\leq r < s \leq k}p_r p_s +(k-1)\sum_{r=1} ^k p_r +\frac{(k-1)(k-2)}{2}.$
%\end{thm}
%
%
%
%\begin{rem}\label{rem1}
%For $\3MP$, i.e $k=3$, we have that $$h=p_1p_2+p_1p_3+p_2p_3+1 \ \ (mod \ 2)$$
%\end{rem}

We finish this section with the a couple of useful propositions.
\begin{prop}\label{3mp}
For any defined $\3MP$ $\langle \alpha,\beta ,\gamma \rangle$ of weight $(n,m,k)$, and any defining system $A$, one has:
\begin{enumerate}
  \item $\langle \alpha,\beta ,\gamma \rangle= c(A)+\alpha H^k+\gamma H^n.$
  \item $0\in \langle \alpha,\beta ,\gamma \rangle \Leftrightarrow c(A) \in \alpha H^k+\gamma H^n \Leftrightarrow \langle \alpha,\beta ,\gamma \rangle=\alpha H^k+\gamma H^n.$
\end{enumerate}
\end{prop}

\begin{proof}
The second claim easily follows from the first. The first claim follows easily from the $\Z/p\Z$ linearity of the differential $\delta$.
%consider two defining systems $A,B$ for $\langle \alpha,\beta ,\gamma \rangle$.
%Now compute: $c(A)-c(B)=\alpha a_{\beta \gamma}+(-1)^{n+1} a_{\alpha \beta}\gamma - \alpha b_{\beta \gamma}+(-1)^{n+1} b_{\alpha \beta}\gamma=\alpha (a_{\beta \gamma}-b_{\beta \gamma})+(-1)^{n+1} (a_{\alpha \beta}-b_{\alpha \beta})\gamma$.
%But clearly $\delta (a_{\beta \gamma}-b_{\beta \gamma})=\delta(a_{\alpha \beta}-b_{\alpha \beta})=0$, and the claim is clear.
\end{proof}

\begin{prop}\label{p}
Let $F$ be a field of characteristic prime to $p$, $L$ a prime to $p$ closure of $F$, and $\langle \alpha, \beta, \gamma \rangle$, a defined $\3MP$ of weight $(n,k,m)$ over $F$. Then,
$$0\in \langle \alpha, \beta, \gamma \rangle \Leftrightarrow 0\in \langle \alpha_L, \beta_L, \gamma_L \rangle.$$
\end{prop}

\begin{proof}
The $(\Rightarrow)$ direction is clear. For the other direction, consider an element $c\in \langle \alpha, \beta, \gamma \rangle$. By assumption after extending scalars to $L$ we have $c_L=\alpha_L\cdot h + g\cdot \gamma_L$ for $h\in H^{m+k-1}(L); \ g\in H^{n+m-1}(L)$, hence there exist a field extension,
$F \subseteq K \subseteq L$ of finite dimension, which is prime to $p$, such that $c_K=\alpha_K\cdot h + g\cdot \gamma_K$ for $h\in H^{m+k-1}(K); \ g~\in~H^{n+m-1}(K).$
The Lemma follows by taking corestriction using the projection formula and \Pref{3mp}.
\end{proof}

%\section{Triple Massey products and symbols}
%In this section we deal with $\3MP$ with weight $(n,k,m)$ where some of the slots are assumed to be symbols, that is elements of the form $\alpha=a_1\cdots a_n$ where $a_i\in H^1(F)$.
\section{ A Lemma on the kernal of multiplication by an $H^1$ element with a symbol}
Most of the material in this section is taken (with minor modifications) from \cite{Weib}, \cite{SJ}. %[Suslin and.. ].
Let $F$ be a field of characteristic zero, %different from $p$,
which is prime to $p$ closed-i.e $F$ has no field extensions of degree prime to $p$, and in particular $F$ contains a primitive $p$-th root of unity $\rho$.
%Recall that an element $\alpha \in H^n(F)$ is called a symbol if it is of the form $\alpha=a_1\cdots a_n$ where $a_i\in H^1(F)$.
Symbols turn out to be important in Galois cohomology as by the Bloch-Kato conjecture proved by Voevodsky and Rost they generate the mod $p$ cohomology algebra of $F$.
One of their main features is the fact that they have "$p$-generic splitting varieties" generalizing Severi-Brauer varieties, which help in understanding their presentations.
The main lemma we are going to prove is the following:
\begin{lem}\label{L1}
Let $F$ be a field of characteristic prime to $p$, which is prime to $p$ closed. Let $\alpha \in H^n(F)$ be a symbol and $b\in H^1(F)$. Then $\alpha \cdot b=0$ if and only if there exist $s_i\in H^1(F) \  i=1,\dots, n$ and a presentation $\alpha=s_1\cdots s_n$ such that $s_n\cdot b=0$.
\end{lem}
In order to prove Lemma \ref{L1} we recall some facts about symbols.
Let $\alpha=a_1\cdots a_n \in H^n(F)$ be a non trivial symbol, then there exist a generic splitting variety of dimension $p^{n-1}-1$ for $\alpha$, namely a smooth, irreducible, projective $F$ variety $X$ of dimension $p^{n-1}-1$, such that for any field extension $L/F$ one has that $\alpha_L=0$ if and only if $X(L)\neq \phi$, for a detailed construction of such $X$ we refer the reader to \cite{Weib}, \cite{SJ}.
\newpage
\begin{thm}[Rost]\label{11}
Let $\alpha \in H^n(F)$, $n\geq 2$, be a non trivial symbol. Then:
\begin{enumerate}
  \item There exist a a geometrically irreducible, smooth projective generic splitting variety, $X$, for $\alpha$ of dimension $p^{n-1}~-~1$.
  \item Every element of $\bar{A_0}(X,K_1)$ is of the form $[x,\lambda]$ where $x$ is a closed point of $X$ of degree $p$ and $\lambda$ is an element of $F(x)$.
\end{enumerate}

\end{thm}
\begin{thm}[\cite{SJ} Theorem A.1 combined with Proposition 2.9]\label{22}
Let $X$ be a generic splitting variety of dimension $p^{n-1}-1$ for a non zero symbol $\alpha \in H^n(F)$. Then the following sequence is exact:
$$0\rightarrow \bar{A_0}(X,K_1)\overset{N}{\longrightarrow} F^{\times} \overset{\cdot \alpha}{\longrightarrow}H^{n+1}(F)$$
where $N$ is the norm map taking $[x,\lambda]$ to $\N_{F(x)/F}(\lambda)$.
\end{thm}
Finally we have the following Theorem:
\begin{thm}[\cite{SJ} Theorem 5.6] \label{33}
Let $E_1, \cdots, E_n$ be a sequence of cyclic splitting fields of
degree $p$ for a non-trivial symbol $\alpha \in H^n(F)$. Then there exist a sequence $s_1, \cdots, s_n\in F^{\times}$ such that:
\begin{enumerate}
  \item $\alpha=s_1\cdots s_n$
  \item For every $1\leq i \leq n$ one has that $E_i$ splits $s_1\cdots s_i\in H^i(F)$
\end{enumerate}
\end{thm}
We are now ready to prove Lemma \ref{L1}\\
\noindent \emph{proof of Lemma 3.1.}\\%\begin{proof}~of Lemma 3.1
In section $6$, we show that the finite characteristic case follows from the characteristic zero case, hence we assume $\ch(F)=0$. The $(\Leftarrow)$ direction is clear. For the other direction, let $\alpha\in H^n(F)$ be a non zero symbol and $b\in F^{\times}$ such that $\alpha \cdot b=0 \in H^{n+1}$. Then, by the exact sequence of Theorem \ref{22} we get that $b$ lies in the image of the norm map. By Theorem \ref{11} we get that there is an element $[x,\lambda]\in \bar{A_0}(X,K_1)$ such that $F(x)=F[\sqrt[p]{s}]$ is a splitting field of $\alpha$ of degree $p$ and $N([x,\lambda])=b$, equivalently $s\cdot b=0$. Finally applying Theorem \ref{33} to the splitting field $F(x)$ yields the existence of a sequence $s_1, \cdots, s_n\in F^{\times}$ such that $\alpha~=~s_1\cdots s_n$ with $s_1~=~s$  and the Lemma follows.\indent \indent \indent \indent \indent \indent \indent \indent \indent\indent \indent \indent \indent \indent\indent\indent\indent\indent\indent\indent\indent\indent \ \ \ \ \ \ \ \ \ \ \ \ \ $\Box$
% \end{proof}
\newpage
\section{Applications of \Lref{L1}}
As a first application of \Lref{L1} we recall and then generalize the following useful Lemma of Tignol (see \cite{Tig1} Lemma 1).

\begin{lem}[Tignol]\label{Tignol}
Let $F$ be a field with $\ch(F)\neq 2$ and $K~=~F[\sqrt{d}]$ a field extension of degree $2$. Let $A=(a,k)_{2,K}$ be a quaternion algebra over $K$, where $a\in \mul{F}$ and $k\in \mul{K}$. %with Galois group $\langle \s_K \rangle$
Then, $$A=\res_{K/F}((a,c)_{2,F}) \ ; c\in \mul{F} \hbox{ \ if and only if \ } \Cor_{K/F}(A)=F \hbox{ \ in \ } \Br(F).$$
\end{lem}
%\begin{proof} This is well known but for completeness we include a proof.
%Clearly if $A$ is a restricted quaternion then $\Cor_{K/F}(A)=F$ in $\Br(F)$. For the other direction, assume $\Cor_{K/F}(A)=F$ in $\Br(F)$.
%Now, %in $\Br(F)$ we have $F=\Cor_{K/F}(A)=\Cor_{K/F}((a,k)_{2,K})=(a,\N_{K/F}(k))_{2,F}$. Hence
%in $\Br(K)$ we have $K=\res_{K/F}\Cor_{K/F}(A)=\Tr(A)=(a,k)_{2,K}+(a,\s(k))_{2,K}=(a,k)_{2,K}-(a,\s(k))_{2,K}=(a,k/\s(k))_{2,K}.$ Thus, by Wedderburn  (\cite{Rowen}, Remark 24.26) there exist an element $u\in L=KM$ where $M=F[\sqrt{a}]$, such that $\N_{L/K}(u)=k/\s(k)$, which implies that $N_{L/F}(u)=1$. Notice that $L/F$ is Galois with group $\langle \s_K ,\s_M \rangle$. By Hilbert $90$ there exist an element $m\in M$ such that, $\s(m)/m=\N_{L/M}(u)$. Then, by \cite{AS}, $L$ and the elements $m,k,u$ define an abelian crossed product
%$B=L[z_1,z_2 \ | \ z_1^2=k, z_2^2=m, z_1l=\s_M(l) z_1, z_2l=\s_K(l)z_2, z_1z_2=uz_2z_1]$. Notice that $\res_{K/F}(B)=L[z_1]=(a,k)_{2,K}=A$.
%But as $2B=\Cor_{K/F}\res_{K/F}(B)=\Cor_{K/F}(A)=F$ in $\Br(F)$, a Theorem of Tignol on decomposition groups extending a Theorm of Albert on bi-quaternions (\cite{Tig} Corollary 2.8 see also \cite{Row} for a similar statement) tells us that
%$B=(d,s)_{2,F}\otimes (a,c)_{2,F}$ from which we clearly get, $A=\res_{K/F}(B)=\res_{K/F}((a,c)_{2,F})$, as needed.
%\end{proof}
We generalize it in the following way,
\begin{thm}\label{Tig}
Let $F$ be a field with $\ch(F)\neq 2$, and $K/F$ a field extension of degree $2$. Let $\alpha=\alpha'\cdot k\in H^{n+1}(K,\Z/2\Z)$ where, $\alpha'~\in~H^n(F)$ is a symbol and $k\in H^1(K)$. Then, $$\alpha=\res_{K/F}(\alpha'\cdot c) \ ; c\in H^1(F) \ \hbox{if and only if} \ \Cor_{K/F}(\alpha)=0.$$
\end{thm}

\begin{proof}
Again in section $6$, we show that the finite characteristic case follows from the characteristic zero case, hence we assume $\ch(F)~=~0$. First notice that if $L/F$ is a field extension of odd degree, then  $$\alpha~=~\res_{K/F}~(\alpha'~\cdot~ c)~ \ ~ ;~c~\in~H^1(F)  \ \hbox{if and only if}  \ \alpha_{KL}=\res_{KL/L}(\alpha'_{L}\cdot c') \ ; c'\in H^1(L).$$ Indeed, one direction is clear and the other follows from the fact that $\alpha'\in H^n(F)$, $[L:F]$ is odd and the projection formula.
Also we clearly have  $$\Cor_{K/F}(\alpha)=0 \ \hbox{if and only if} \ \Cor_{KL/L}(\alpha)=0$$ since $[L:F]$ is odd. Thus we may restrict scalars to a prime to $2$ closure, $L$ of $F$. %that $F$ is prime to $2$ closed.
The first direction of the Lemma $(\Rightarrow)$ is clear. For the other direction, assume $\Cor_{K/F}(\alpha)=0$ in $H^{n+1}(F)$. Going up to $L$ we have,
$0~=~\Cor_{KL/L}~(\alpha)=\alpha'_L\cdot \N_{KL/L}(k)$, hence by Lemma \ref{L1} we see that $\alpha'_L=\alpha'' \cdot s$ where $\alpha'' \in H^{n-1}(L)$ a symbol and $s~\cdot~\N_{KL/L}~(k)~=~0$. Now consider the quaternion algebra $A=(s,k)_{2,KL}$, clearly we have $\Cor_{KL/L}(A)~=~(~s~,~\N_{KL/L}~(k))_{2,L}=L$ in $\Br(L)$. Thus, by \Lref{Tignol} we get that
$A=\res_{KL/L}(s,t) \ ; \ t\in \mul{L}$. Combining things we get, $\alpha_{KL}=\alpha'' \cdot s\cdot k=\alpha'' \cdot s\cdot t=\alpha'_{KL}\cdot t$ and as mentioned above this implies that $\alpha_{K}=\res_{K/F}(\alpha' \cdot d) \ ; \ d\in \mul{F}$ as needed.
\end{proof}
%\section{$\3MP$ with symbol entries}
A second application of Lemma \ref{L1} will be to $\3MP$ of weight $(n,1,m)$ where the first and last entries are symbols.
\begin{thm}\label{n1m}
Let $F$ be a field of characteristic prime to $p$, and $\langle \alpha, a, \beta \rangle$ be a defined $\3MP$ of weight $(n,1,m)$ where $\alpha, \beta$ are symbols. Then, $~0~\in~\langle \alpha, a, \beta \rangle.$
\end{thm}
\begin{proof}
%First recall that defined $\3MP$ of weight $(1,1,1)$ always contain zero ().
Using \Pref{p} we may assume $F$ is prime to $p$ closed. Now, by definition we have $\alpha \cdot d=d\cdot \beta=0$. Thus Lemma \ref{L1} yields presentations, $\alpha=a_1\cdots a_n; \ \beta=b_1\cdots b_m$ such that $a_n\cdot d=0$ and $d\cdot b_1=0$. Hence the $\3MP$ of weight $(1,1,1)$, $\langle a_n, d, b_1 \rangle$ is defined and contains zero by the main Theorem of \cite{EM2},\cite{Mat14}.
Now using \Tref{KT1} we see that: $0\in (-1)^{n-1}a_1\cdots a_{n-1}\langle a_n, d, b_1 \rangle b_2\cdots b_n\subseteq \langle \alpha, d, b_1 \rangle b_2\cdots b_n\subseteq \langle \alpha, d, \beta \rangle$ and we are done.
\end{proof}
%We will switch between the cohomology ring and the Milnor K-ring freely.

%For this subsection we assume $p=2$, i.e. $H^i=H^i(G_F,\Z/2\Z)$, unless otherwise stated.
\section{$\3MP$ and multiplication by a symbol}
We start this section by recalling the description of the kernel of multiplication by a symbol, denoted $\Ker(\alpha)$, given in \cite{OVV} for the case~$p=2$ and in \cite{MS} for an arbitrary prime.
Let $F$ be a field of characteristic zero, and~$\alpha =\{a_1,\dots, a_n \}\in k_n(F)=K^M _n(F)/2K^M _n(F) $ a symbol.
Define $Q_\alpha$ to be the projective quadric of dimension $2^{n-1}-1$ defined by the form $q_\alpha=\langle \langle a_1,\dots,a_{n-1}\rangle \rangle - \langle a_n \rangle$.
This quadric is called the small Pfister quadric or the norm quadric associated with the symbol $\alpha$ and it is a generic splitting variety for $\alpha$. Denote by $F(Q_\alpha)$ the function field of $Q_\alpha$ and by $(Q_\alpha)_{(0,\leq 2)}$ the set of closed points of $Q_\alpha$ of degree at most~$2$.
\begin{thm}[Orlov, Vishik, Voevodsky]
Let $F$ be a field of characteristic zero, and $\alpha \in K^M _n(F) $ a symbol. Then for every $i\geq 0$, the following sequence is exact,
$$\oplus_{x\in (Q_\alpha)_{(0,\leq 2)}} k_{i}(F(x)) \xrightarrow{\sum \Tr_{F(x)/F}} k_i(F) \xrightarrow{\cdot \alpha} k_{i+n}(F)\xrightarrow{\res_{F(Q_{\alpha})/F}} k_{i+n}(F(Q_{\alpha})).$$
\end{thm}

Now, recalling the fact that for a field extension $L/F$ of degree $2$, $k_n(L)$ is generated by symbols of the form $\{f_1, \dots, f_{n-1},l\}$, where $f_1, \dots, f_{n-1}\in F^{\times}$ and $l\in L^{\times}$
%(Bass and Tate? H. Bass and J. Tate, The Milnor ring of a global field, Lecture Notes in Math. 342 (1973), 349-446. Corollary 5.3)
one obtains the following variant of Theorem $3.3$ in \cite{OVV}:
\begin{thm}\label{1}
Let $F$ be a field of characteristic zero, and $\alpha \in K^M _n(F) $ a symbol.
The Kernel of multiplication by $\alpha$, $\Ker(\alpha)$, is generated as an abelian group by all elements of the form $\beta\cdot f$, where $\beta$ is a symbol and $f\in \N_{L/F}(L^{\times})$, where $L$ runs over all splitting fields of $\alpha$ of degree at most $2$. That is $$\Ker(\alpha)=\sum_{L} k(F) \cdot \N_{L/F}(L^{\times})$$ where $L$ is as above and $k(F)$ is the mod $2$ Milnor $K$-ring of $F$.
\end{thm}

For an arbitrary prime $p$, one has the following description of the kernel of multiplication by a symbol (see \cite{MS} for more details).
\begin{thm}[Merkurjev, Suslin]\label{2}
Let $F$ be a field of characteristic prime to $p$ and $\theta\in H_{et} ^n(F,\mu_p^{\otimes n})$ a symbol, where $\mu_p$ denotes the Galois module of all $p$-th roots of unity. Then for an arbitrary $k\in \mathbb{N}$, there is an exact sequence
$$\coprod_L H_{et} ^k(L,\mu_p^{\otimes k})  \xrightarrow{\sum \N_{L/F}} H_{et} ^k(F,\mu_p^{\otimes k})\xrightarrow{\cdot \theta} H_{et} ^{k+n}(F,\mu_p^{\otimes k+n})\xrightarrow{ \prod_{E} \res_{E/F}}\prod_E H_{et} ^{k+n}(E,\mu_p^{\otimes k+n})$$
where the coproduct is taken over all finite splitting field extensions $L/F$ for $\theta$ and the product is taken over all splitting field
extensions $E/F$.\end{thm}

%\begin{rem}
%In Theorem \ref{2} we can not find an argument to show it is enough to restrict to splitting fields of degree at most $p$. If one can prove that, one wi
%\end{rem}

\subsection{$\3MP$ where the middle entry is a symbol} \ \\
We are now ready to apply the above to $\3MP$ where the middle slot is a symbol.
Let $F$ be a field of characteristic zero, and $p$ an arbitrary prime.

\begin{defn}
A basic element of $Ker(\alpha)$ is an element of the form $\{f_1, \dots, f_{n-1},f\}$ , where $f_1, \dots, f_{n-1}\in F^{\times}$, $f\in \N_{L/F}(L^{\times})$ and $L$ is a degree $p$ splitting field for $\alpha$.
\end{defn}

\begin{thm}\label{Basic}
Let $F$ be a field with $\ch(F)\neq p$ and $\langle x,\alpha, y \rangle $ be a defined $\3MP$ where, $\alpha \in H^k$ is a symbol and $x\in H^n; y\in H^m$ are basic elements of $\Ker(\alpha)$. Then $0\in \langle x,\alpha, y\rangle.$
\end{thm}

\begin{proof}
By Proposition \ref{p}, we may assume $F$ is prime to $p$ closed.
As $x,y$ are basic we can write $x=x'\cdot r; y=s\cdot y'$ where $x'~\in~H^{n-1};y'~\in~H^{m-1}$ are symbols and $r\in N_{L/F}(L^{\times}); s\in \N_{T/F}(T^{\times})$ for $L,T$ splitting fields for $\alpha$ of degree $p$. The first thing to notice it that $r,s\in \Ker(\alpha)$, hence the $\3MP$ $\langle r,\alpha, s \rangle $ is defined.
We now split to cases: The first case is when $p$ is an odd prime. In this case by Corollary  3.4 of \cite{3MPwithW} we have $0\in \langle r,\alpha, s \rangle$. Thus by Theorem \ref{KT1} we get,
$$0\in (-1)^{n-1}x'\langle r,\alpha, s \rangle y' \subseteq (-1)^{2(n-1)}\langle x'\cdot r,\alpha, s\cdot y' \rangle=\langle x,\alpha, y\rangle$$ and we are done.
Now consider the case $p=2$. As in the first case it will be enough to show $0\in \langle r,\alpha, s \rangle$.
Note that as $r\cdot \alpha =0$, Lemma \ref{L1} says that we can write $\alpha=a\cdot \alpha'$ where $a\in H^1; \ \alpha'\in H^{k-1} \hbox{a symbol}$ and $r\cdot a=0$.
Thus there exist $\varphi_{ra} \in C^1$ such that $\delta(\varphi_{ra})=r\cdot a$. Now define $\varphi_{r\alpha}=\varphi_{ra}\cdot \alpha'\in C^{k}$ and compute $\delta(\varphi_{r\alpha})=ra\alpha'=r\alpha$.
Also as $\alpha s=0$ there exist $\varphi_{\alpha s} \in C^k$ such that $\delta(\varphi_{\alpha s})=\alpha s $. Collecting the above we constructed a defining system for the $\3MP$ $\langle r,\alpha, s \rangle $ with value $A=\varphi_{r\alpha}\cdot s +r\cdot \varphi_{\alpha s}$. We will now use $A$ to prove  $0\in \langle r,\alpha, s \rangle$.
First, letting $K$ be the field extension corresponding to $r$ under the Kummer map, we notice that by construction $\res_{\ker(r)}(\varphi_{ra})$ is an element of $H^1(K)$, indeed $\delta_{\ker(r)} (\varphi_{r a})=(r\cdot a)_{\ker(r)}=0$.
Now consider $B_K=\res_{K/F}(A)=\res_{K/F}(\varphi_{r\alpha}\cdot s +r\cdot \varphi_{\alpha s})=(\varphi_{ra})_{K}\cdot \alpha'_{K}\cdot s_{K}$, and compute $\Cor_{K/F} (B)=\Cor_{K/F}\res_{K/F}(A)=2A=0$.
Thus applying \Tref{Tig} we get that $B_K=\res_{K/F}(d\cdot\alpha'\cdot s)$ and we write $B_F=d\cdot\alpha'\cdot s\in H^k\cdot s$.
Now as $\res_{K/F}(A-B)=0$, we get that $A-B\in r\cdot H^k$ from which we conclude that $A\in r\cdot H^k+H^k\cdot s$ and we are done.
\end{proof}

\begin{cor}\label{1n1}
Let $F$ be a field of characteristic prime to $p$, and $\langle x,\alpha, y \rangle $ be a defined $\3MP$ of weight $(1,n,1)$ where $\alpha$ is a symbol. Then~,~$0~\in~\langle x,\alpha, y\rangle.$
\end{cor}
\begin{proof}
By \Pref{p} we may go up to a prime to $p$ closure, $L/F$ of $F$. Over $L$, \Tref{22} implies that $x,y$ are basic elements of $\Ker(\alpha)$ so we are done by \Tref{Basic}.
\end{proof}
\begin{rem}
Notice that for an odd prime $p$ \Cref{1n1} is weaker than Corollary $3.4$ of \cite{3MPwithW} where it is proved that any defined $\3MP$ of weight $(1,k,1)$ contains zero.
\end{rem}

%\begin{cor}\label{cor}
%Let $p$ be an arbitrary prime, $\alpha \in H^n; \ \gamma\in H^m$ be symbols and $b\in H^1$. If the $\3MP$ $\langle \alpha,b,\gamma \rangle$ is defined, then it contains zero.
%\end{cor}
%\begin{proof}
%Note that by Lemma \ref{L1} we see that $\alpha, \gamma$ are basic, hence we are done by Theorem \ref{Basic}.
%\end{proof}

\begin{defn}
For a symbol $\alpha$ define $\BKer(\alpha)$ to be the subgroup of $\Ker(\alpha)$ generated by all basic elements of $Ker(\alpha)$.
\end{defn}

\begin{rem}
Note that for $p=2$, $\BKer(\alpha)=\Ker(\alpha)$ by Theorem \ref{1}. For arbitrary prime $p$ this holds at least for $\alpha \in H^1$ by \Tref{22}.
The author believes one can prove that over a prime to $p$ closed field one can prove that $\BKer(\alpha)=\Ker(\alpha)$ for an arbitrary prime.
\end{rem}
\newpage
\begin{thm}\label{BKer}
Let $F$ be a field of characteristic prime to $p$, $\beta\in H^k$ a symbol, and $\alpha\in H^n\cap \BKer(\beta), \gamma\in H^m\cap \BKer(\beta)$. Then the 3MP $\langle \alpha, \beta, \gamma \rangle$ is defined and for any presentation $\alpha=\sum_{i=1} ^{s} \alpha_i; \gamma=\sum_{j=1} ^{t} \gamma_j$ where the summands are basic elements of $\Ker(\beta)$ we have,
$$\langle \alpha, \beta, \gamma \rangle \subseteq \sum_{i,j}(\alpha_i\cdot H^{k+m-1}+\gamma_j\cdot H^{n+k-1}).$$
\end{thm}

\begin{proof}
The fact that the $\3MP$ is defined is follows from the exact sequence of Theorem \ref{2}, further note that for every relevant $i,j$ we have that the $\3MP$ $\langle \alpha_i, \beta, \gamma_j \rangle$ is also defined. Thus for every $i,j$ there exist a defining system $C_{i,j}=\{\varphi_{\alpha_i\beta},\varphi_{\beta \gamma_j}\}$ for $\langle \alpha_i, \beta, \gamma_j \rangle$ with values $c_{i,j}~=~\alpha_i~\cdot~\varphi_{\beta \gamma_j}~+~\varphi_{\alpha_i\beta}~\cdot~\gamma_j$. Notice that by Theorem \ref{Basic} we have $0~\in~\langle~\alpha_i,~\beta,~\gamma_j~\rangle$, hence we have $c_{i,j}~\in~\alpha_i~\cdot~H^{k+m-1}~+~\gamma_j~\cdot~H^{n+k-1}$. Now define $\varphi_{\alpha \beta}=\sum \varphi_{\alpha_i\beta}$, $\varphi_{\beta \gamma}=\sum \varphi_{\beta \gamma_j}$ and notice that $C=\{\varphi_{\alpha \beta}, \varphi_{\beta \gamma}\}$ is a defining system for $\langle \alpha, \beta, \gamma \rangle$ with value $c=\alpha\cdot \varphi_{\beta \gamma}+\varphi_{\alpha \beta}\cdot \gamma=\sum_{i,j}(\alpha_i\cdot \varphi_{\beta \gamma_j}+\varphi_{\alpha_i\beta}\cdot \gamma_j)=\sum_{i,j}c_{i,j}\in \sum_{i,j}(\alpha_i\cdot H^{k+m-1}+\gamma_j\cdot H^{n+k-1}).$ Finally, as the indeterminacy $\alpha\cdot H^{k+m-1}+\gamma\cdot H^{n+k-1}$ is a subgroup of $\sum_{i,j}(\alpha_i\cdot H^{k+m-1}+\gamma_j\cdot H^{n+k-1})$, we obtain:
$$\langle \alpha, \beta, \gamma \rangle=c+\alpha\cdot H^{k+m-1}+\gamma\cdot H^{n+k-1}\subseteq \sum_{i,j}(\alpha_i\cdot H^{k+m-1}+\gamma_j\cdot H^{n+k-1}),$$ and we are done.
\end{proof}

%\begin{rem}
%It seems that the existence of $p$-generic splitting varieties was proved for fields of characteristic zero but it is implied that things work for fields of characteristic different for $p$ (see \cite{SJ} middle of page 246). Also in \cite{MS}%(Mer Sus: MOTIVIC COHOMOLOGY OF THE SIMPLICIAL MOTIVEOF A ROST VARIETY]
%they are used for fields of characteristic different for $p$. Hence most likely, all the above is valid for these fields as well.
%\end{rem}

\section{The case of characteristic prime to $p$}

Let $F$ be a field of characteristic $q>0$ prime to $p$. In this section we prove that results \Lref{L1} (assuming $F$ is also prime to $p$ closed), \Tref{Tig} and \Tref{n1m} are also valid over $F$. The main idea is to reduce to the characteristic zero case by building a mixed characteristic complete local ring, having $F$ as its residue.
We have the following Theorem taken from [\cite{Matsumura}, section 29].%We start with further assuming that $F$ is finitely generated as a field over $\Fq$, which we can always assume in our situation as every object we deal with is defined by a finite number of elements. The following should be standard, we include a proof for completeness.
\begin{prop}\label{model}
There exist a complete local ring $R$, with respect to a discrete valuation $\mu$, (associated to the maximal ideal $pR$) of characteristic zero, having $F$ as its residue.
\end{prop}
%\begin{proof}
%By assumption there is a projection: $p_1: \Fq[x_1,\dots, x_n]\to F$, with kernel $J=\langle f_1, \dots f_t \rangle$ where the $f_i$'s are polynomials with $\Fq$ coefficients. Lifting them to polynomials $F_1,\dots, F_t$ over $\Z$ we can construct $Q=\Z[x_1,\dots, x_n]/ \langle F_1, \dots, F_t \rangle$, and get a projection $p_3:Q\to F$ with kernel $I=qQ$. Then $I$ is a maximal ideal and we can define $R$ to be the completion of $Q$ with respect to the $I$-adic topology. Then $R$ has a discrete valuation $\mu$ defined by the maximal ideal $I=\langle q \rangle$, and residue field $F$.
%\end{proof}

Let $\pr:R\to F$ be the natural projection and let $K$ be the fraction field of $R$. Then, the maximal ideal of $R$ defines a discrete valuation, $\mu$, on $R$ which extends to $K$. Since $K$ is complete with respect to $\mu$, Hensel's Lemma allows one to lift (via $\pr$) Galois extensions of $F$ to unramified Galois extensions of $K$, and we know (see for example \cite[page 52 corollary of proposition 3.3]{LF} or \cite[appendix, Theorem A.23]{TW}) that there is an isomorphism $i:\Gal(F^{sep}/F)\cong \Gal(K^{ur}/K)$. This isomorphism yields an isomorphism of the differential graded $\Zp$-algebras of continues cochains, \begin{equation} \label{eq1} i:C(\Gal(K^{ur}/K),\Zp)\to C(\Gal(F^{sep}/F), \Zp),\end{equation} inducing an isomorphism of their cohomology algebras.
After inflating we get a map: $\beta: H(\Gal(F^{sep}/F), \Zp)\to H(\Gal(K^{sep}/K),\Zp)$.
Recall (for more details see \cite[chapter 7]{GS}) that we have a ramification map associated with $\mu$:
$$\ram :H^n(K,\Zp) \to H^{n-1}(F,\Zp)$$ such that $$\ram(u_1\cdots u_{n-1}\cdot \pi)=\pr(u_1)\cdots \pr(u_{n-1})$$ where $u_i$'s are units in $R$ and $\pi$ is a fixed uniformizer for $\mu$.

Given a sequence of elements $a_1,\dots, a_t\in H(\Gal(K^{ur}/K),\Zp)$, let $T=\sum a_iH(\Gal(K^{sep}/K),\Zp)$ and $T^{ur}=\sum a_iH(\Gal(K^{ur}/K),\Zp)$.
We have the following lemma:
\begin{lem}\label{ur}
Let $\alpha \in H(\Gal(K^{ur}/K),\Zp)$ be such that after Inflation to $H(\Gal(K^{sep}/K),\Zp)$ one has, $\alpha\in T$. Then, $\alpha\in T^{ur}$.
\end{lem}
\begin{proof}
If $\alpha=0$ we are done, so we assume $\alpha\neq 0$. %Clearly it is enough to prove it for homogeneous elements.
Now, by assumption we can write $\alpha=\sum a_ib_i$ where $a_i\in H^{n_i}(\Gal(K^{ur}/K),\Zp)$ and $b_i\in H^{m_i}(\Gal(K^{sep}/K),\Zp)$. It will be enough to show that $b_i\in H^{m_i}(\Gal(K^{ur}/K),\Zp)$.
Assume the $b_i$'s are ramified, so we can write {$b_i=c_i+b'_i \cdot\pi$}~where $\pi$ is a fixed uniformizer for $\mu$, $b'_i\in  H^{m_i-1}(\Gal(K^{ur}/K),\Zp)$ is some pre-image of $\ram(b_i)$ under the isomorphism $i$, ((\ref{eq1}) above) and the $c_i$'s are in  $H^{m_i}(\Gal(K^{ur}/K),\Zp)$ such that equality holds. Indeed such $c_i$ always exist because of the isomorphism $i$, ((\ref{eq1}) above) and the fact that $b_i-b'_i \pi$ is in the kernel of the ramification map. Now, since $\alpha$ is unramified we get that $0=\ram(\alpha)=\sum \pr(a_i\cdot b'_i)$ so again via the isomorphism $i$ above we see that $0=\sum a_ib'_i$ implying that $$\alpha=\sum a_i(c_i+b'_i \cdot\pi)=\sum a_ic_i+\sum a_ib'_i \cdot\pi=\sum a_ic_i,$$ and we are done.
\end{proof}
\newpage
We also have the following technical Lemma:
\begin{lem}\label{tec}
Fix a uniformizer $\pi$ for $\mu$. Every symbol $\beta\neq 0$ over $K$ has a presentation of the form $\beta=u_1\cdots u_{n-1}\cdot (u_n\pi^{\epsilon})$ where $u_i$'s are units in $R$, $\epsilon\in \{0,1\}$.
Moreover, for any such presentation $\epsilon=0$ if and only if $\beta$ is unramified.
\end{lem}
\begin{proof}
Start with a presentation $\beta= s_1\cdots s_n$. If all slots are unramified (so may be assumed to be units) or only one slot is ramified, the first part of the Lemma is clear. If more than one slot is ramified, say  $s_1=u_1\pi^i; s_2=u_2\pi^j$ (we may assume that by the anti-commutativity of cup product), then $s_1\cdot s_2=(u_1\pi\cdot u_2\pi)^{ij}=(u_1/u_2)^{ij}\cdot u_2\pi$ and now the first slot is a unit and the first part of the Lemma follows. For the second part, let $\beta=u_1\cdots u_{n-1}\cdot (u_n\pi^{\epsilon})$ be a presentation as in the Lemma. If $\epsilon=0$, $\beta$ is unramified. If $\epsilon=1$ then, $\ram(\beta)= \pr(u_1)\cdots \pr(u_{n-1})$ is non trivial and $\beta$ is ramified. Indeed, if $\pr(u_1)\cdots \pr(u_{n-1})=0$, then we get that $u_1\cdots u_{n-1}=0$ (by the isomorphism of the unramified cohomology and that of the residue field) and $\beta$ is trivial contrary to our assumption.
\end{proof}

\begin{cor}\label{ctec}
Fix a uniformizer $\pi$ for $\mu$ and let $\beta$ be a non trivial symbol over $K$. If $\beta$ is unramified, then for every presentation, $\beta~=~s_1\cdots s_n$, all slots are unramified, as elements of $H^1$.
\end{cor}

\begin{cor}\label{gen}
Let $F$ be a field of characteristic prime to $p$, which is prime to $p$ closed, $\alpha=a_1\cdots a_n \in H^n(F, \Zp)$ and $b\in H^1(F,\Zp)$. Then $\alpha \cdot b=0$ if and only if there is a presentation $\alpha=s_1\cdots s_n$ such that $s_n\cdot b=0$ (i.e \Lref{L1} is also valid for fields of characteristic prime to $p$).
\end{cor}
\begin{proof}
Let $\alpha, b$ be as above. The characteristic zero case is \Lref{L1}. For the case of characteristic prime to $p$, construct the complete local ring $K$ as above and lift $\alpha, b$ to $i(\alpha), i(b)\in H(\Gal(K^{ur}/K), \Zp)$.
Then, $i(\alpha)\cdot i(b)=0$ since $i$ is an isomorphism. Thus by the characteristic zero case we have that $i(\alpha) = s_1\cdots s_n$ such that $s_n\cdot i(b)=0$. Since $i(\alpha)$ is unramified we get from \Cref{ctec} that all slots are unramified, so may be assumed to be units. Hence $\alpha=\pr(s_1)\cdots \pr(s_n)$ and $\pr(s_n)\cdot b=0$ as needed.
%If $s_n$ is ramified we can use it to make the other slots of $\alpha$ unramified as in the proof of \Lref{tec}, so we may assume $s_1, \dots, s_{n-1}$ are unramified. But as $i(\alpha)$ is unramified, \Lref{tec} implies that $s_n$ is unramified contrary to assumption. Thus $s_n$ is unramified. Now by \Lref{tec} we may write $s_1\cdots s_{n-1}=u_1\cdots u_{n-2}\cdot (u_{n-1}\pi^{\epsilon})$ to get $i(\alpha)=u_1\cdots u_{n-2}\cdot (u_{n-1}\pi^{\epsilon})\cdot s_n$ and again by \Lref{tec} and the fact that $i(\alpha)$ is unramified we conclude that $\epsilon=0$. Finally we see that all the slots of this new presentation of $i(\alpha)$ are unramified, so we can go back to $F$ via the isomorphism $i$ (\ref{eq1} above), and get that $\alpha=\pr(s_1)\cdots \pr(s_n)$ and clearly, $\pr(s_n)\cdot b=0$ as needed.
\end{proof}
\newpage
\begin{cor}
\Tref{n1m} and \Tref{Tig} are also valid for a field, $F$, of characteristic prime to $p$.
\end{cor}
\begin{proof}
\Tref{Tig} is clear in light of \Cref{gen} and \Pref{p}. As for \Tref{n1m}, we only need to show that if over $F$ the $\3MP$ $\langle \alpha, a, \beta\rangle$ is defined with a defining system $\varphi$ then lifting to $K$ the $\3MP$ $\langle i(\alpha), i(a), i(\beta)\rangle $ is defined over $K$ with defining system $i(\varphi)$. But this is clear as the isomorphism $i$ is an isomorphism of the differential graded algebra of continuous cochains and we are done.
\end{proof}

\end{document}